\documentclass[12pt]{article}
\usepackage[english]{babel}
\usepackage[latin1]{inputenc}
\usepackage[T1]{fontenc}
\usepackage{amsmath,amssymb,amsthm,xspace}

\RequirePackage{ifpdf}
\ifpdf
   \usepackage[pdftex]{hyperref}
\else
   \usepackage[hypertex]{hyperref}
\fi

\title{The isometry group of the bounded Urysohn space is simple
}
\author{Katrin Tent and Martin Ziegler}
\date{November 20, 2012}

\newtheorem{theorem}{Theorem}[section]

\newtheorem{lemma}[theorem]{Lemma}
\newtheorem{proposition}[theorem]{Proposition}
\newtheorem{corollary}[theorem]{Corollary}
\newtheorem{definition}[theorem]{Definition}

\newcommand{\nc}{\newcommand}
\nc{\inv}{^{-1}}
\nc{\Q}{\mathbb{Q}}
\nc{\R}{\mathbb{R}}
\nc{\C}{\mathcal{C}}
\nc{\G}{\mathcal{G}}
\nc{\M}{\mathcal{M}}
\nc{\U}{\mathbb{U}}

\nc{\Frl}{Fra\"iss\'e limit\xspace}
\nc{\Frls}{Fra\"iss\'e limits\xspace}
\nc{\fg}{finitely generated\xspace}

\renewcommand{\phi}{\varphi}

\DeclareMathOperator{\Isom}{Isom}

\DeclareMathOperator{\tp}{tp}

\newcommand{\Ind}{
 \setbox0=\hbox{$x$}\kern\wd0\hbox to 0pt{\hss$
 \mid$\hss}\lower.9\ht0\hbox to 0pt{\hss$\smile$\hss}\kern\wd0
}
\newcommand{\indep}[3]{#1\mathop{\mathpalette\Ind{}}_{#2}#3
}

\begin{document}
\maketitle
\begin{abstract}We  show that the isometry group of the bounded Urysohn space is a simple group.
\end{abstract}

\section{Introduction}

The bounded Urysohn space $\U_1$ of diameter $1$ is the (unique)
complete homogeneous separable metric space of diameter $1$ which
embeds every finite metric space of diameter $1$.  It was shown in
\cite{TZ2} that the isometry group of the (general) Urysohn space
modulo the subgroup of bounded isometries is a simple group and it was
widely conjectured (in particular by M. Rubin and J. Melleray) that
the isometry group of the bounded Urysohn space is a simple group.  We
here prove this conjecture using the approach from \cite{TZ2}:

\begin{theorem}\label{t:main}
The isometry group of $\U_1$ is abstractly simple.
\end{theorem}

Note that we cannot expect bounded simplicity as in the results in \cite{TZ2}
as there are isometries of $\U_1$ with arbitrarily small displacement.

The proof relies on the properties of an abstract independence relation.
We will continue to use the concepts introduced in \cite{TZ2}, in particular
the following notion of independence:

\begin{definition}
We say that $A$ and $C$ are independent over $B$, written \[\indep{A}{B}{C},\]
if for all $a\in A,c\in C$ with $d(a,c)<1$ there is some $b\in B$
such that $d(a,c)=d(a,b)+d(b,c)$.

We say that an automorphism $g\in \Isom(\U_1)$ moves \emph{almost maximally}
if for all types $\tp(a/X)$ with $X$ finite there is a realisation $b$
with \[\indep{b}{X}{g(b)}.\]

\end{definition}
Note that this definition of independence makes sense even if
$B=\emptyset$ and hence this defines a stationary independence
relation in the sense of \cite{TZ2}.  The proof here follows the same
lines as the proof in \cite{TZ2} and we will continue using notions
from that paper. In the next section we will establish the following:

\begin{proposition}
  Let $g\in \Isom(\U_1)$. If $d(a,g(a))=1$ for some $a\in\U_1$, then a
  product of $2^5$ conjugates of $g$ moves almost maximally and hence
  any element of $\Isom(\U_1)$ can be written as the product of $2^9$
  conjugates of $g$ and $g\inv$.
\end{proposition}

Using the following observation, this proposition will then imply Theorem~\ref{t:main} exactly as in \cite{TZ2}.

\begin{lemma}\label{l:move1}
If $g\in\Isom(\U_1)$ is not the identity, then a product of conjugates of
$g$ moves some element by distance $1$.
\end{lemma}
\begin{proof}
Let $a\in\U_1$ be such that $d(a,g(a))=k>0$. Pick $b\in U_1$ with $d(a,b)=1$
and a sequence of elements $a_0=a,\ldots a_m=b$ with $d(a_{i-1},a_i)=k, i=1,\ldots m$.
By homogeneity of $\U_1$ there are elements $h_i\in\Isom(\U_1), i=1,\ldots m$ with $h_i(a)=a_{i-1}, h_i(g(a))=a_i$. Then $g^{h_i}(a_{i-1})=a_i$
and hence the product of these conjugates moves $a$ to $b$.
\end{proof}

\section{Proof of the main result}

For any finite set $X\subset \U_1,a\in\U_1$ we write $d(a,X)=\min
\{d(a,x)\colon x\in X\}$ for the distance from $a$ to $A$. We call
$d(a,X)$ also the distance of the type $\tp(a/A)$.  We put $G=\Isom(\U_1)$.

\begin{lemma}\label{l:1-sphere}
Let $g\in G$ be such that for some $a\in\U_1$ we have $d(a,g(a))=1$.
Then for any finite set $A$ there is some $x$ with $d(x,A)=1$ and
$d(x,g(x))\geq~1/2$.
\end{lemma}
\begin{proof}
Clearly we may assume that $a\in A$. Put $Y=A \cup g\inv(A)$ and choose
some $b$ with $d(b,a)=1/2$ and independent from $Y$ over $a$.
Then $d(g(b),A)\geq 1/2$ and  since $d(a,g(a))=1$ we also have $d(g(b),a)=1$.
Therefore we have $d(b,g(b))\geq 1/2$.
Choose $x$ with $d(x,Ab)=1$ such that $d(x,g(b))$ is minimal. Since $d(g(b),Ab)\geq 1/2$,
we have $ d(x,g(b))\leq  1/2$ and hence $d(x,g(x))\geq 1/2$.
\end{proof}

Let $p=\tp(a/X)$ be a type over a finite set $X$.  We say that $g\in
G$ moves the type $p$ \emph{almost maximally} if there is a
realisation $x$ of $p$ with $\indep{x}{X}{g(x)}$ and it moves the type
$p$ by distance $C$ if there is a realisation $x$ of $p$ with
$d(x,g(x))\geq C$.

\begin{lemma}\label{l:alltypes}
  Let $g\in G$ and $1\geq d_0\geq 0$ be such that $g$ moves any type
  of distance $d_0$ almost maximally.  Then any type of distance $d\leq
  d_0$ is moved almost maximally or by distance $1-2(d_0-d)$.
\end{lemma}

\begin{proof}
  Let $p=\tp(x/X)$ be a type of distance $d\leq d_0$ and $x'$ a
  realisation of $p$ independent from $g\inv(X)$ over $X$ (so
  $d(x',Xg\inv(X))=d$). Put $p'=\tp(x'/Xg\inv(X))$ and let $q=p'+(d_0-d)$
  denote the prolongation of $p'$ by $d_0-d$.

  By assumption on $g$, there is a realisation $z$ of $q$ which is
  moved almost maximally over $Xg\inv(X)$.  Hence
  \[\indep{z}{Xg\inv(X)}{g(z)}\]
  and by transitivity
  \[\indep{z}{X}{g(z)}.\]

  If $d(z,g(z))=1$ then for a realisation $y$ of $p'$ with
  $d(y,z)=d_0-d$ we clearly have $d(y,g(y))\geq 1-2(d_0-d)$.\\

  Otherwise we find some $b\in X$ such
  that \[d(z,g(z))=d(z,b)+d(b,g(z)).\] Let $y$ be a realisation of
  $p'$ with $d(y,z)=d_0-d$.  Note that by definition of the
  prolongation we have \[\indep{z}{y}{Xg\inv(X)}\mbox{ \ \ and hence
    \ \ }\indep{g(z)}{g(y)}{X}.\]
  Therefore \[d(z,g(z))=d(z,y)+d(y,b)+d(b,g(y))+d(g(y),g(z))\] and in
  particular
\[\indep{y}{X}{g(y)}.\]

\end{proof}

\begin{lemma}\label{l:2kd}
  Let $g\in G$. Then there exists some $h\in G$ such that $[g,h]$ has the
  following property for all $d$ and $C$: if $g$ moves all types of
  distance $d$ almost maximally or by distance $C$, then $[g,h]$ moves
  all types of distance $d$ almost maximally or by distance $2C$.
\end{lemma}
\begin{proof}
  As in \cite{TZ2} we may work in a countable model of the bounded
  Urysohn space.  We build $h$ by a `back-and-forth' construction as
  the union of a chain of finite partial automorphisms. It is enough
  to show the following: let $h'$ be already defined on the finite set
  $U$, let $p$ be a type over $X$ of distance $d$ and assume that $g$
  moves all such types almost maximally or by distance $C$. Then $h'$
  has an extension $h$ such that $[g,h]$ moves $p$ almost maximally or
  by distance $2C$.

  We may assume that $X$ is contained in $U$. We denote by $V$ the
  image of $U$ under $h'$. Consider any realisation $a$ of $p$
  independent from \[Y=Ug\inv(U)g\inv(X)\] and a realisation $b$ of
  $h'(\tp(a/U))$ over $V$. Then we extend $h'$ to $h:Uac\cong Vbg(b)$
  where $c$ is a realisation of $h\inv(\tp(g(b)/Vb))$ independent from
  $Xg(a)$. Then $a$ is moved under $[g,h]$ to $g\inv(c)$. Since
  \[\indep{c}{Ua}{g(a)}\mbox{ \ and \ \ } \indep{g(a)}{g(X)}{UX}\] we have
  $\indep{c}{g(X)a}{g(a)}$, which means
  that \[\indep{c}{a}{g(a)}\mbox{ \ or \ }\indep{c}{g(X)}{g(a)}.\] The
  second case implies $\indep{g^{-1}(c)}{X}{a}$, which implies our
  claim.\\

  Since $d(a/Y)=d(a/U)=d$, our assumption about $d$ and $C$ implies
  that one of the following three cases occur:\\

  \noindent {\bf Case 1.} We find $a$ and $b$ as above with $d(a,g(a))\geq
  C$ and $d(b,g(b))\geq C$. By the above we may assume that
  $\indep{c}{a}{g(a)}$. If $d(c,g(a))=d(g^{-1}(c),a)=1$, then $g^{-1}(c)$
  and $a$ are independent over the empty set and hence over $X$. Otherwise we have
  \[d(g^{-1}(c),a)=d(c,g(a))=d(c,a)+d(a,g(a))=d(b,g(b))+d(a,g(a))\geq 2C.\]

  \noindent  {\bf Case 2.} $\indep{a}{Y}{g(a)}$: This implies
  $\indep{a}{X}{g(a)}$. Since $\indep{g(a)}{g(X)}{X}$ transitivity yields
  $\indep{a}{g(X)}{g(a)}$.  So from $\indep{c}{ag(X)}{g(a)}$, then we
  get $\indep{c}{g(X)}{g(a)}$ and hence $\indep{g\inv(c)}{X}{a}$ as
  desired.\\

  \noindent  {\bf Case 3.} $\indep{b}{V}{g(b)}$: This implies
  $\indep{a}{U}{c}$.  As above we now
  get \[\indep{c}{g(X)}{g(a)}\mbox{ \ \ and hence
    \ \ }\indep{g\inv(c)}{X}{a}.\]

\end{proof}

By the results in \cite{TZ2} we now obtain:
\begin{proposition}\footnote{We thank Adriane Ka\"ichouh and Isabel
    M\"uller for pointing
    out an error in an earlier version of the proposition.}\label{C:simple}
  Let $g\in\ G$ be such that for some $a\in\U_1$ we have
  $d(a,g(a))=1$.  Then every element of $G$ is the product of $2^9$
  conjugates of $g$ and $g\inv$.
\end{proposition}

\begin{proof}
  An iterated application of Lemma~\ref{l:2kd} to $g$ yields
  isometries $g_1$, $g_2$, $g_3$,$g_4$ and $g_5$. Note that $g_5$ is a
  product of $2^5$ conjugates of $g$ and $g^{-1}$.\\

  \noindent By Lemma~\ref{l:1-sphere} $g$ moves every type with
  distance $1$ by distance $\frac{1}{2}$. So $g_1$ moves every type of
  distance $1$ almost maximally or by distance $2\cdot 1/2=1$, hence
  almost maximally.  Now Lemma~\ref{l:alltypes} (with $d_0=1$) implies
  that $g_1$ moves every type of distance $d$ almost maximally or by
  distance $1-2(1-d)=2d-1$.\\

  \noindent This implies that $g_2$ moves every type of distance $d$
  almost maximally or by distance $4d-2$. So types of distance $d\geq
  \frac{3}{4}$ are moved almost maximally and using
  Lemma~\ref{l:alltypes} with $d_0=\frac{3}{4}$ we see that types of
  distance $d\leq\frac{3}{4}$ are moved almost maximally or by
  distance $1-2(\frac{3}{4}-d)=2d-\frac{1}{2}$.\\

  \noindent Now $g_3$ moves every type of distance $d$ almost
  maximally or by distance $4d-1$. So types of distance $d\geq
  \frac{1}{2}$ are moved almost maximally and using
  Lemma~\ref{l:alltypes} with $d_0=\frac{1}{4}$ we see that types of
  distance $d\leq\frac{1}{2}$ are moved almost maximally or by
  distance $1-2(\frac{1}{2}-d)=2d$.\\

  \noindent This implies that $g_4$ moves every type of distance $d$
  almost maximally or by distance $4d$. So types of distance $d\geq
  \frac{1}{4}$ are moved almost maximally and using
  Lemma~\ref{l:alltypes} with $d_0=\frac{1}{4}$ we see that types of
  distance $d\leq\frac{1}{4}$ are moved almost maximally or by
  distance $1-2(\frac{1}{4}-d)=2d+\frac{1}{2}$.\\

  \noindent So $g_5$ moves all types almost maximally.
  By Corollary 5.4 in \cite{TZ2}, every element of $G$ is a product of
  at most $2^4$ conjugates of $g_5$ or its inverse.
\end{proof}

\begin{corollary}
  Let $g\in G$. If there is $a\in\U_1$ with $d(a,g(a))\geq 1/n$, then
  any element of $G$ can be written as a product of at most $n\cdot
  2^9$ conjugates of $g$ and $g\inv$.
\end{corollary}



\vspace{2cm}

\noindent\parbox[t]{15em}{
Katrin Tent,\\
Mathematisches Institut,\\
Universit\"at M\"unster,\\
Einsteinstrasse 62,\\
D-48149 M\"unster,\\
Germany,\\
{\tt tent@uni-muenster.de}}
\hfill\parbox[t]{18em}{
Martin Ziegler,\\
Mathematisches Institut,\\
Albert-Ludwigs-Universit\"at Freiburg,\\
Eckerstr. 1,\\
D-79104 Freiburg,\\
Germany,\\
{\tt ziegler@uni-freiburg.de}}

\end{document}